\documentclass[11pt, oneside, reqno]{amsart}
\usepackage{amsfonts, amssymb, color}
\usepackage{amsmath}
\usepackage{graphicx}
\usepackage{amscd}
\usepackage{amsthm}
\usepackage{tikz}
\usetikzlibrary{decorations.markings}
\usetikzlibrary{shapes,arrows,positioning}	
\usetikzlibrary {arrows.meta}	
\usetikzlibrary{decorations.markings,decorations.pathmorphing}

\newcommand{\rmidarrow}{\tikz \draw[thick,-{Computer Modern Rightarrow}] (-1pt,0) --(1pt,0);}	
\newcommand{\lmidarrow}{\tikz \draw[thick, -{Computer Modern Rightarrow}] (1pt,0) -- (-1pt,0);}

\newcommand{\Nwmidarrow}{\tikz \draw[thick, -{Computer Modern Rightarrow}] (0,0) -- (2pt,2pt);}
\newcommand{\Nmidarrow}{\tikz \draw[thick, -{Computer Modern Rightarrow}] (0,0) -- (0,1);}
\newcommand{\Nemidarrow}{\tikz \draw[thick, -{Computer Modern Rightarrow}] (0,0) -- (-2pt, 2pt);}
\usepackage[margin=1in]{geometry}

	\newtheorem{thm}{Theorem}[section]
	\newtheorem{Observation}{Observation}[thm]
	\newtheorem{lem}[thm]{Lemma}
	\newtheorem{prop}[thm]{Proposition}
	\newtheorem{cor}[thm]{Corollary}
	
	\newtheorem{rem}[thm]{Remark}
	\newtheorem{rems}[thm]{Remarks}
	\newtheorem{Def}[thm]{Definition}

\begin{document}
	
\title[]{Quantum Symmetries of Graph $C^*$-algebras Having Maximal Permutational Symmetry}
\author{Ujjal Karmakar \MakeLowercase{and} Arnab Mandal}

\maketitle

\begin{abstract}
	     Quantum symmetry of a graph $C^{*}$-algebra $C^{*}(\Gamma)$ corresponding to a finite graph $\Gamma$ has been explored by several mathematicians within different categories in the past few years. In this article, we establish that there are exactly three families of compact matrix quantum groups, containing the symmetric group on the set of edges of the underlying graph $\Gamma$, that can be achieved as the quantum symmetries of graph $C^*$-algebras in the category introduced by Joardar and Mandal. 
	     Moreover, we demonstrate that there does not exist any graph $C^*$-algebra associated with a finite graph $\Gamma$ without isolated vertices having $A_{u^t}(F^{\Gamma})$ as the quantum automorphism group of $C^*(\Gamma)$ for a non-scalar matrix $F^{\Gamma}$.\\
\end{abstract}
 
 
 \section{Introduction}
Compact matrix quantum groups, introduced by S.L. Woronowicz in 1987 \cite{Wor}, can be viewed as a generalization of symmetries for some non-commutative spaces. In 1995, A. Connes posed the problem of identifying an appropriate notion of the quantum automorphism group for non-commutative spaces to extend the concept of classical group symmetry to a `non-commutative version of symmetry' \cite{connes}. A few years later, S. Wang formulated the notion of a quantum symmetry group for $X_n$, a space containing $n$ points, in a categorical language \cite{Wang}. While the classical automorphism group of $X_n$ is simply the permutation group $S_n$, he showed that the underlying $C^*$-algebra of the quantum automorphism group is infinite-dimensional and non-commutative for $n \geq 4$ (see \cite{Wang}).
 Between 2003 and 2005, J. Bichon and T. Banica extended the quantum symmetry structure for a simple finite graph \cite{Ban, Bichon}. Since then, extensive studies on the quantum symmetries of graphs have been conducted in \cite{Banver}, \cite{Lupini}, \cite{SSthesis}, \cite{Fulton}, \cite{Qauttrees}, \cite{lexico} etc. The notion of a quantum symmetry group for continuous structures within an infinite-dimensional framework of non-commutative geometry was formulated by D. Goswami in \cite{Laplace}.\\

 On the other hand, a graph C*-algebra, an extension of the Cuntz-Krieger algebra \cite{Cuntz, CK}, is a universal C*-algebra generated by orthogonal projections and partial isometries associated with a given directed graph. Notably, a large family of $C^*$-algebras, including matrix algebras, Cuntz algebras, Toeplitz algebra, quantum spheres, quantum projective spaces and more, can be realized as graph $C^*$-algebras. More interestingly, various operator algebraic properties of $C^*$-algebras can be derived from the underlying graph and vice versa.
 It is worth mentioning that a graph $C^*$-algebra can be recognized as a Cuntz-Krieger algebra if the underlying finite graph has no sinks (see \cite{Kumjian}, \cite{Tom}, \cite{Raeburn}).\\
 
 In 2017, S. Schmidt and M. Weber came up with the notion of the quantum symmetry of a graph $C^*$-algebra $C^*(\Gamma)$ by showing that the quantum automorphism group of its underlying graph $\Gamma$ in the sense of Banica, denoted by $QAut_{Ban}(\Gamma)$, acts maximally on $C^{*}(\Gamma)$ \cite{Web}. Later, S. Joardar and A. Mandal defined the quantum symmetry of a graph $C^*$-algebra considering a category with an analytic nature  \cite{Mandal}. They showed in \cite{Mandal} that the quantum symmetry group of $C^*(\Gamma)$, denoted by  ${Q}_{\tau}^{Lin}(\Gamma)$, is always a quantum subgroup of $A_{u^t}(F^{\Gamma})$ for an invertible diagonal matrix $F^{\Gamma}$ of order $|E(\Gamma)|$, where $E(\Gamma)$ denotes the edge set of $\Gamma$. Moreover, ${Q}_{\tau}^{Lin}(\Gamma)$ always contains the quantum group $ (\underbrace{C(S^{1})* \cdots *C(S^{1})}_{|E(\Gamma)| \textendash times}, \Delta)$, and we have already produced a class of graph $C^*$-algebras that exhibit minimal quantum symmetry (i.e., isomorphic to $ (\underbrace{C(S^{1})* \cdots *C(S^{1})}_{|E(\Gamma)| \textendash times}, \Delta)$) in \cite{rigidity}. Similarly, one can raise the following question:
 \begin{itemize}
 	\item When does ${Q}_{\tau}^{Lin}(\Gamma)$ achieve the maximal quantum symmetry $A_{u^t}(F^{\Gamma})$? \\
 \end{itemize} 
\noindent However, we have shown in Theorem \ref{AUtF} that there does not exist any graph $C^*$-algebra whose quantum symmetry group is (identical to) $A_{u^t}(F^{\Gamma})$ for a non-scalar matrix $F^{\Gamma}$. Hence, for a given graph $\Gamma$, the only possible choice for maximal quantum symmetry is the free unitary group $U_{|E(\Gamma)|}^+$. Additionally, Theorem 4.12 of \cite{Mandal} ensures that $U_n^+$ appears as the quantum symmetry of graph $C^*$-algebra $C^*(L_n)$ associated with the graph $L_n$ (see Figure \ref{Ln}). Let us now introduce the following terminology for better exposition. A graph $C^*$-algebra $C^*(\Gamma)$ has {\bf maximal permutational symmetry} if its quantum symmetry group ${Q}_{\tau}^{Lin}(\Gamma)$ contains the symmetric group on the edge set of $\Gamma$. More precisely, $C^*(\Gamma)$ has maximal permutational symmetry if $S_{|E(\Gamma)|}$ is a quantum subgroup of ${Q}_{\tau}^{Lin}(\Gamma)$.  Since $S_{n}$ is a quantum subgroup of $U_{n}^+$, it follows from the above discussion that $C^*(L_n)$ has maximal permutational symmetry. This observation leads to the following question:
\begin{itemize}
	\item Can we classify all graph $C^*$-algebras having maximal permutational symmetry and identify their corresponding quantum symmetry groups?\\
\end{itemize} 
In the context of the quantum symmetry for a finite graph, since the action of $QAut_{Ban}(\Gamma)$ depends on the vertex set $V(\Gamma)$ of its underlying graph $\Gamma$, similar questions can be framed as follows: (a) Can one identify all the graphs $\Gamma$ such that their respective quantum symmetry groups are isomorphic to $S_{|V(\Gamma)|}^+$ (the maximal quantum symmetry)? (b) Can we classify all the graphs $\Gamma$ whose quantum symmetry groups $QAut_{Ban}(\Gamma)$ contain the classical symmetry groups $S_{|V(\Gamma)|}$?  In this scenario, there are exactly two graphs with $n$ vertices, namely $K_n$ (complete graph with $n$ vertices) and $K_n^{c}$ (graph with $n$ isolated vertices), for which $QAut_{Ban}(K_n) \cong QAut_{Ban}(K_n^{c}) \cong S_n^+$ (see \cite{Banver}). Moreover, these are the only two graphs whose quantum symmetry group $S_n^+$ contains $S_n$. In this article, we classify all the graph $C^*$-algebras which admit maximal permutational symmetry.\\
 
 Now, we briefly discuss the presentation of this article as follows: In the second section, the prerequisites related to graph $C^*$-algebras, compact quantum groups, the action of a compact quantum group on a $C^*$-algebra, and quantum automorphism groups are covered. Additionally, we revisit the quantum symmetry of a graph $C^*$-algebra with some definitions and proven results within the category introduced in \cite{Mandal}. In the third section, using some auxiliary lemmas and propositions, we identify all the graph $C^*$-algebras having maximal permutational symmetry, as summarized in Theorem \ref{ET1} and Table \ref{Table}. Finally, Theorem \ref{AUtF} establishes that there does not exist any graph $C^*$-algebra associated with a finite graph $\Gamma$ (without isolated vertices) having $A_{u^t}(F^{\Gamma})$ as the quantum automorphism group of $C^*(\Gamma)$ for a non-scalar matrix $F^{\Gamma}$.  
 
\section{Preliminaries}

\subsection{Notations and conventions}
For a set $A$, the cardinality of $A$ will be denoted by $|A|$, and  the identity function on $A$ will be denoted by $id_{A}$. For a $C^*$-algebra $ \mathcal{C}$, $\mathcal{C}^*$  is the set of all linear bounded functionals on $ \mathcal{C} $. For a set $A$, span($A$) will denote the linear space spanned by the elements of $A$. The tensor product `$\otimes$' denotes the spatial or minimal tensor product between $C^*$-algebras. $I_n$ denotes the identity matrix of $M_n(\mathbb{C})$. All the $C^*$-algebras are unital for us. Throughout this article, all the graphs are assumed to be directed graphs.

\subsection{Graph $C^{*}$-algebras}
In this subsection, we will recall some basic definitions regarding directed graphs and preliminary facts about graph $C^{*}$-algebras from \cite{Pask}, \cite{Raeburn}, \cite{BPRS} and \cite{MS}.\\
A directed graph $\Gamma=\{ V(\Gamma), E(\Gamma),s,r \}$ consists of countable sets $ V(\Gamma) $ of vertices and $ E(\Gamma)$ of edges together with the maps $s,r: E(\Gamma) \to V(\Gamma) $ describing the source and range of the edges. For $v,w \in V(\Gamma)$, the notation `$v \to w $' means that there exists an edge $e \in E(\Gamma)$ such that  $ v=s(e)$ and $ w=r(e) $. We say that $v$ and $w$ are {\bf adjacent} (denoted by $v \sim w$) if $v \to w$ or $w \to v$. A graph is said to be \textbf{finite} if both $|V(\Gamma)|$ and $|E(\Gamma)|$ are finite. For $v \in V(\Gamma)$, {\bf indegree} of $v$ (in short, $\mathbf{indeg(v)}$) denotes the cardinality of the set $r^{-1}(v)$. A vertex $v$ is called a {\bf sink} if $s^{-1}(v)= \emptyset$. A vertex $v$ is called a {\bf rigid source} if $r^{-1}(v)= \emptyset$. A \textbf{path} $\alpha$ of length $n$ in a directed graph $\Gamma$ is a sequence $\alpha=e_{1}e_{2} \cdots e_{n} $ of edges in $\Gamma$ such that $ r(e_{i}) = s(e_{i+1}) $ for $1 \leq i \leq (n-1) $. We define $s(\alpha):=s(e_{1})$ and  $r(\alpha):=r(e_{n})$. A \textbf{loop} is an edge $l$ with $s(l)=r(l)$. An edge $e \in E(\Gamma)$ is said to be a {\bf multiple edge} if there exists another $f \in E(\Gamma)$ such that $s(e)=s(f)$ and $r(e)=r(f)$. A directed \textbf{graph without isolated vertices} \footnote{Note that in \cite{Mandal}, the terminology ``connected" was used in the sense of a graph without isolated vertices.} means for every vertex $v \in V(\Gamma) $ at least one of $s^{-1}(v)$ or $r^{-1}(v)$ is non-empty. A directed graph is said to be \textbf{connected} \footnote{In this article, authors use the term ``connected" in the usual graph-theoretic sense, which is different from \cite{Mandal}.} if given any two distinct vertices $v$ and $w$, either $v \sim w$ or there exists a finite sequence of vertices $v_1, v_2,...,v_n$ such that $v \sim v_1 \sim v_2 \sim \cdots \sim v_n \sim w$.\\
Let $\Gamma=\{V(\Gamma),E(\Gamma),s,r \}$ be a finite, directed graph with $|V(\Gamma)|=n$. The adjacency matrix of $\Gamma$ with respect to the ordering of the vertices $ (v_{1},v_{2},..., v_{n}) $ is a matrix $ (a_{ij})_{i,j= 1,2,...,n} $ with 
$ a_{ij} =
\begin{cases}
n(v_{i},v_{j}) & if ~~ v_{i} \to v_{j} \\
0 &  ~~ otherwise 
\end{cases} $ where $ n(v_{i},v_{j})$ denotes the number of edges with source $v_{i}$ and range $v_{j}$.\\ 

\noindent In this article, we will define the graph $C^*$-algebra only for a finite, directed graph.  
\begin{Def}
	Given a finite, directed graph $\Gamma$, the graph $C^{*}$-algebra $C^{*}(\Gamma)$ is a universal $C^{*}$-algebra generated by partial isometries $ \{S_{e}: e \in E(\Gamma) \} $ and orthogonal projections $ \{p_{v}: v \in V(\Gamma) \}$ such that \vspace{0.1cm}
	\begin{itemize}
		\item[(i)] $S_{e}^{*}S_{e}=p_{r(e)}$ for all $ e \in E(\Gamma) $, \vspace{0.1cm}
		\item[(ii)] $p_{v}=\sum\limits_{\{f:s(f)=v\}}S_{f}S_{f}^{*}$ for all $ v \in V(\Gamma) $ if $s^{-1}(v) \neq \emptyset $.\\
	\end{itemize}
\end{Def} 

 \noindent For a given graph $C^{*}$-algebra, we have the following properties which will be used later.
\begin{prop} \label{properties}
	For a finite, directed graph $\Gamma=\{V(\Gamma),E(\Gamma),s,r \}$,
	\begin{itemize}
		\item[(i)] $S_{e}^{*}S_{f}=0$  for all $e \neq f $.\vspace{0.1cm}
		\item[(ii)] $\sum\limits_{v \in V(\Gamma)}p_{v}=1$.\vspace{0.1cm}
		\item[(iii)] $ S_{e}S_{f}\neq 0 \Leftrightarrow r(e)=s(f)$ i.e., $ef \text{ is a path} $.     \\
		Moreover, $S_{\gamma}:=S_{e_{1}}S_{e_{2}}...S_{e_{k}} \neq 0 \Leftrightarrow r(e_{i})=s(e_{i+1}) $ for $i=1,2,...,(k-1)$ i.e., $ \gamma=e_{1}e_{2}...e_{k}$ is a path.\vspace{0.1cm}
		\item[(iv)] $ S_{e}S_{f}^{*}\neq 0 \Leftrightarrow r(e)=r(f). $\vspace{0.1cm}
		\item[(v)] If $\gamma=e_1e_2...e_k$ is a path and $S_{\gamma}:=S_{e_{1}}S_{e_{2}}...S_{e_{k}}$, then $S_{\gamma}^* S_{\gamma}=p_{r(e_k)}$. \vspace{0.1cm}
		\item[(vi)] $ span\{S_{\gamma}S_{\mu}^{*} : \gamma, \mu \in E^{< \infty}(\Gamma) \text{ with } r(\gamma)=r(\mu)\} $ is dense in $C^{*}(\Gamma)$, where $E^{< \infty}(\Gamma)$ denotes the set of all finite length paths.\\
	\end{itemize} 
\end{prop}
 
\noindent The lemma below discusses the independence of the partial isometries corresponding to each path of length 2.
\begin{lem}
	The set $ \mathcal{E}^{(2)} := \{S_e S_f : e,f \in E(\Gamma) \text{ with } r(e)=s(f)\}$ is a linearly independent set. \label{independent}
\end{lem}
\begin{proof}
	For fixed $e,f \in E(\Gamma)$ with $r(e)=s(f)$, let
	\begin{align*}
	& \sum\limits_{\{g,h: r(g)=s(h)\}} c_{gh}S_g S_h=0 .\\
	\end{align*}
Multiplying both sides of the above equation by $S_e^*$ from the left, we get	
	\begin{align*}
& 	\sum\limits_{\{g,h: r(g)=s(h)\}} c_{gh} S_e^* S_g S_h=0\\
\Rightarrow & 	\sum\limits_{\{h: r(e)=s(h)\}} c_{eh} p_{r(e)} S_h=0 ~~[\text{by Proposition } \ref{properties}~(i)]\\
\Rightarrow & \sum\limits_{\{h: s(h)=r(e)\}} \sum\limits_{\{k: s(k)=r(e)\}} c_{eh} S_k S_k^* S_h=0\\
\Rightarrow & \sum\limits_{\{h: s(h)=r(e)\}}  c_{eh} S_h S_h^* S_h =0 ~~[\text{by Proposition } \ref{properties}~(i)]\\
\Rightarrow & \sum\limits_{\{h: s(h)=r(e)\}}  c_{eh} S_h =0.\\
\end{align*}
Again, multiplying both sides of the above equation by $S_f^*$ from the left, we have
\begin{align*}
 & \sum\limits_{\{h: s(h)=r(e)\}}  c_{eh} S_f^* S_h =0\\
\Rightarrow &~~   c_{ef} S_f^* S_f = c_{ef}~ p_{r(f)}=0\\
\Rightarrow &~~ c_{ef}=0.
	\end{align*} 
Hence, $\mathcal{E}^{(2)}$ is a linearly independent set.	
\end{proof}
\subsection{Compact quantum groups }
In this subsection, we will recall some important aspects of compact matrix quantum groups and their actions on a general $C^{*}$-algebra. We refer the readers to \cite{Van}, \cite{Wang}, \cite{Wor}, \cite{webcqg}, \cite{remark}, \cite{Tim}, \cite{Nesh} and \cite{Fres} for more details.
\begin{Def} \label{CMQG}
	Let $A$ be a unital $C^*$-algebra and $q=(q_{ij})_{n \times n}$ be a matrix whose entries are coming from $A$. $(A,q)$ is called a compact matrix quantum group (CMQG) if the following conditions are satisfied:
	\begin{enumerate}
		\item[(i)] A is generated by the entries of the matrix $q$.
		\item[(ii)] Both $q$ and ${q}^t:=(q_{ji})_{n \times n}$ are invertible in $M_n(A)$.
		\item[(iii)] There exists a $C^*$-homomorphism $\Delta: A \to A \otimes A $ such that $\Delta(q_{ij})= \sum_{k=1}^{n}q_{ik} \otimes q_{kj}$ for all $i,j \in \{1,2,...,n\}$. 
	\end{enumerate}
The matrix $q$ is called the fundamental representation of $A$. 
\end{Def}
\begin{Def} \label{identical}
	Two CMQGs $(A,q)$ and  $(A',q')$, where $A$ and $A'$ are unital $C^*$-algebras with fundamental representations $q=(q_{ij})_{n \times n}$ and $q'=(q_{ij}')_{n \times n}$ respectively, are said to be { \bf identical} (denoted by $(A,q) \approx (A',q')$) if there exists a $C^*$-isomorphism $ \phi: A \to A' $ such that $ \phi(q_{ij})=q_{ij}'$.
\end{Def}
\begin{Def}
	A CMQG $(B,q^{B})$ is said to be a quantum subgroup of $(A,q^{A})$ if there exists a surjective $C^*$-homomorphism $ \phi: A \to B $ such that $ \phi(q^{A}_{ij})=q^{B}_{ij}$.
\end{Def}
\begin{Def} \label{CQG}
	A compact quantum group (CQG) is a pair $(\mathcal{Q}, \Delta) $, where $\mathcal{Q}$ is a unital $C^{*}$-algebra and the coproduct $ \Delta : \mathcal{Q} \to \mathcal{Q} \otimes \mathcal{Q} $  is a unital $C^{*}$-homomorphism such that \vspace{0.1cm}
	\begin{itemize}
		\item[(i)] $(id_{\mathcal{Q}} \otimes \Delta)\Delta = (\Delta \otimes id_{\mathcal{Q}})\Delta $, \vspace{0.1cm}
		\item[(ii)] $ span\{ \Delta(\mathcal{Q})(1 \otimes \mathcal{Q} )\}$ and $ span\{\Delta(\mathcal{Q})(\mathcal{Q} \otimes 1)\} $ are dense in $(\mathcal{Q}\otimes \mathcal{Q})$. \vspace{0.2cm}
	\end{itemize}
Given two CQGs $(\mathcal{Q}, \Delta)$ and $(\mathcal{Q}', \Delta') $, a compact quantum group morphism (CQG morphism) from $(\mathcal{Q}, \Delta)$ to $(\mathcal{Q}', \Delta')$ is a $C^{*}$-homomorphism $ \phi: \mathcal{Q} \to \mathcal{Q}' $ such that $ (\phi \otimes \phi)\Delta=\Delta' \phi $ .\\
\end{Def}

\begin{rems}
 (1) Since $q^t=\bar{q}^*$, one can replace the condition (ii) of Definition \ref{CMQG} by ``both $q$ and $\bar{q}:=(q_{ij}^*)_{n \times n}$ are invertible in $M_n(A)$".\vspace{0.2cm}
 
 \noindent (2) In a CMQG $(A,q)$, let $A_0$ denote a *-subalgebra generated by $q_{ij}$'s. Then there exists a linear anti-multiplicative map $\kappa: A_0 \to A_0$, called the antipode, such that $\kappa(\kappa(a^*))^*=a$ for all $a \in A_0$ and the matrix $q_{\kappa}:=(\kappa(q_{ij}))_{n \times n}$ is the inverse of $q$.\vspace{0.2cm}
 
\noindent  (3) For any CQG $(\mathcal{Q}, \Delta)$, there also exists a canonical Hopf $*$-algebra $\mathcal{Q}_{0} \subseteq  \mathcal{Q} $ which is dense in $\mathcal{Q}$. Moreover, one can define an antipode $\kappa$ and a counit $\epsilon$ on the dense  Hopf $*$-algebra $\mathcal{Q}_{0}$. A CQG $(\mathcal{Q}, \Delta)$ is said to be {\bf kac type} if $\kappa$ is $*$-preserving.  \vspace{0.2cm}
 
 \noindent (4) One can show that every CMQG is  a CQG.\\
\end{rems}

\noindent We sometimes denote the CQG (CMQG) by $G$ and the underlying $C^*$-algebra associated to it by $C(G)$. In other words, $(C(G),\Delta)$ and $(C(G),q)$ represent the CQG and CMQG respectively. \\

\noindent Next, we will look at some examples of CQGs that appear as CMQGs. Interestingly, a few of them are unitary easy quantum groups (see \cite{speicher}, \cite{easy} for further details regarding easy quantum groups).  \\

\noindent Examples:
\begin{enumerate}
        \item Consider the universal $C^{*}$-algebra generated by $\{u_{ij}\}_{i,j=1,2,...,n}$ such that the following relations are satisfied: \vspace{0.1cm}
	\begin{itemize}
		\item  $u_{ij}^{2}=u_{ij}=u_{ij}^{*} $ for all $i,j \in \{1,2,...,n\}$, \vspace{0.1cm}
		\item  $\sum_{k=1}^{n} u_{ik}= \sum_{k=1}^{n} u_{kj}=1 $ for all $i,j \in \{1,2,...,n\}$, \vspace{0.1cm}
	\end{itemize}
    and denote it by $C(S_n^+)$.\\  
	Moreover, the coproduct on generators is given by $\Delta(u_{ij})= \sum_{k=1}^{n}u_{ik} \otimes u_{kj}$. Then $(C(S_n^+), \Delta)$ forms a CQG, namely the quantum permutation group on $n$ symbols, and it is denoted by $S_n^+$.\\
	Additionally, $C(S_n^+)$ is non-commutative if and only if $n \geq 4$, and 
	$$C(S_n)=C(S_n^+)/<u_{ij}u_{kl}-u_{kl}u_{ij}>$$
	(consult \cite{Wang}, \cite{BBC} for more details).\\

	\item For  $F \in \mathbb{GL}_{n}\mathbb{(C)}$, $ A_{u^{t}}(F)$ be the universal $C^{*}$-algebra generated by  $\{u_{ij}: i,j \in \{1,2,...,n\}\}$ subject to the following relations: \vspace{0.1cm}
	\begin{itemize}
		\item The matrix $u^{t}:=(u_{ji})_{n \times n}$ is unitary. \vspace{0.1cm}
		\item $uF^{-1}u^{*}F=F^{-1}u^{*}Fu=I_{n}$, where $u:=(u_{ij})_{n \times n}$. \vspace{0.1cm} 
	\end{itemize}
	The coproduct is given on the generators $\{u_{ij}\}_{i,j=1,2,...,n}$ by $\Delta(u_{ij})=\sum_{k=1}^{n} u_{ik}\otimes u_{kj} $. One can show that $(A_{u^{t}}(F),\Delta)$ (respectively, $(A_{u^{t}}(F),u)$) is a CQG (respectively, CMQG). In this article, we refer to this CQG (CMQG) again by $A_{u^{t}}(F)$ (consult \cite{Van}, \cite{universalCQG} for details).\\
	 In particular, if $F=\lambda I_n$ for $\lambda \neq 0$, then the associated CMQG $A_{u^{t}}(F)$ is denoted by $U_{n}^{+}$ (see \cite{Wangfree}). $U_n^+$ is called the free unitary group of order $n$. Moreover, $U_n^+$ can be viewed as the unitary easy quantum group with respect to the category of partition $\mathcal{C}_0 =  \emptyset $ (see Example 5.3 of \cite{easy} for more details). \\

	\item  $C(H_{n}^{\infty +})$ is defined to be the universal $C^*$-algebra generated by $\{u_{ij}: i,j \in \{1,2,...,n\}\}$ such that the following relations hold: \vspace{0.1cm}
	\begin{itemize}
		\item The matrices $u:=(u_{ij})_{n \times n}$ and $(u_{ij}^{*})_{n \times n}$ are unitaries. \vspace{0.1cm}
		\item $u_{ij}$'s are normal partial isometries for all $i,j$. \vspace{0.1cm}
	\end{itemize}
	The coproduct $\Delta $ is defined on the generators by $\Delta(u_{ij})=\sum_{k=1}^{n} u_{ik}\otimes u_{kj}$. Then $(C(H_{n}^{\infty +}), u)$ forms a CMQG, and it is denoted by $H_{n}^{\infty +}$ (see \cite{easy} for details). It can be shown that $H_{n}^{\infty +}$ is CQG-isomorphic to $C(S^1) \wr_{*} S_n^+$. We refer the readers to \cite{Wangfree} and \cite{Bichonwreath} for more details about free products and free wreath products. Moreover, the CMQG $H_{n}^{\infty +}$ is unitary easy with respect to the category of partition $\mathcal{C}_0 =\{ \begin{tikzpicture}[scale=0.5]
	\draw[fill=black] (0,0) circle (3pt);
	\draw[fill=black] (0,1) circle (3pt);
	\draw (0.6,0) circle (3pt);
	\draw (0.6,1) circle (3pt);
	\draw[black] (0,0.05)--(0,0.3);
	\draw[black] (0,0.95)--(0,0.7);
	\draw[black] (0.6,0.05)--(0.6,0.3);
	\draw[black] (0.6,0.95)--(0.6,0.7);
	\draw[black] (0,0.3)--(0.6,0.3);
	\draw[black] (0,0.7)--(0.6,0.7);
	\draw[black] (0.3,0.3) -- (0.3,0.7);\end{tikzpicture}, 
	\begin{tikzpicture}[scale=0.5]
	\draw (0,0) circle (3pt);
	\draw (0,1) circle (3pt);
	\draw[fill=black] (0.6,0) circle (3pt);
	\draw[fill=black] (0.6,1) circle (3pt);
	\draw[black] (0,0.05)--(0,0.3);
	\draw[black] (0,0.95)--(0,0.7);
	\draw[black] (0.6,0.05)--(0.6,0.3);
	\draw[black] (0.6,0.95)--(0.6,0.7);
	\draw[black] (0,0.3)--(0.6,0.3);
	\draw[black] (0,0.7)--(0.6,0.7);
	\draw[black] (0.3,0.3) -- (0.3,0.7);\end{tikzpicture}  , \begin{tikzpicture}[scale=0.5]
	\draw (0,0) circle (3pt);
	\draw (0,1) circle (3pt);
	\draw (0.6,0) circle (3pt);
	\draw (0.6,1) circle (3pt);
	\draw[black] (0,0.05)--(0,0.3);
	\draw[black] (0,0.95)--(0,0.7);
	\draw[black] (0.6,0.05)--(0.6,0.3);
	\draw[black] (0.6,0.95)--(0.6,0.7);
	\draw[black] (0,0.3)--(0.6,0.3);
	\draw[black] (0,0.7)--(0.6,0.7);
	\draw[black] (0.3,0.3) -- (0.3,0.7);\end{tikzpicture} \}$  (consult section 4 of \cite{easy} for further details).\\
	
	\item Let $C(SH_{n}^{\infty +})$ denote the universal $C^*$-algebra generated by $\{u_{ij}: i,j \in \{1,2,...,n\}\}$ subject to the following relations: \vspace{0.1cm}
	\begin{enumerate}
		\item The matrices $u:=(u_{ij})_{n \times n}$ and $(u_{ij}^{*})_{n \times n}$ are unitaries. \vspace{0.1cm}
		\item $u_{ij}$'s are partial isometries for all $i,j$. \vspace{0.1cm}
	\end{enumerate}
	Note that condition (b) can be replaced by the alternative condition $`` u_{ik}u_{jk}^*=u_{ik}^*u_{jk}=0 ~~ \forall i,j,k \in \{1,2,...,n\} \text{ with } i \neq j"$.
	The coproduct $\Delta $ on generators is again given by $\Delta(u_{ij})=\sum_{k=1}^{n} u_{ik}\otimes u_{kj}$. Then $(C(SH_{n}^{\infty +}),u)$ forms a unitary easy quantum group with respect to the category of partition $\mathcal{C}_0=\{\begin{tikzpicture}[scale=0.5]
	\draw[fill=black] (0,0) circle (3pt);
	\draw[fill=black] (0,1) circle (3pt);
	\draw (0.6,0) circle (3pt);
	\draw (0.6,1) circle (3pt);
	\draw[black] (0,0.05)--(0,0.3);
	\draw[black] (0,0.95)--(0,0.7);
	\draw[black] (0.6,0.05)--(0.6,0.3);
	\draw[black] (0.6,0.95)--(0.6,0.7);
	\draw[black] (0,0.3)--(0.6,0.3);
	\draw[black] (0,0.7)--(0.6,0.7);
	\draw[black] (0.3,0.3) -- (0.3,0.7);\end{tikzpicture}, 
	\begin{tikzpicture}[scale=0.5]
	\draw (0,0) circle (3pt);
	\draw (0,1) circle (3pt);
	\draw[fill=black] (0.6,0) circle (3pt);
	\draw[fill=black] (0.6,1) circle (3pt);
	\draw[black] (0,0.05)--(0,0.3);
	\draw[black] (0,0.95)--(0,0.7);
	\draw[black] (0.6,0.05)--(0.6,0.3);
	\draw[black] (0.6,0.95)--(0.6,0.7);
	\draw[black] (0,0.3)--(0.6,0.3);
	\draw[black] (0,0.7)--(0.6,0.7);
	\draw[black] (0.3,0.3) -- (0.3,0.7);\end{tikzpicture} \}$ and is simply denoted by $SH_{n}^{\infty +}$.\\	
	\item  The free product of $k (\geq 2)$ copies of $C(S^{1})$, denoted by $( \underbrace{C(S^{1})*\cdots *C(S^{1})}_{k \textendash times}, \Delta )$, is a CQG  whose underlying $C^*$-algebra is the universal $C^*$-algebra generated by $k$ unitary elements $ \{z_{i}\}_{i=1}^{k} $ (i.e., $C^*\{z_{1},z_{2},...,z_{k}~|~ z_{i}z_{i}^{*}=z_{i}^{*}z_{i}=1 ~\forall~ i \in \{1,2,...,k\} \}$), and the coproduct $\Delta$ on $\{z_{i}\}_{i=1}^{k}$  is given by $\Delta(z_{i})=z_{i} \otimes z_{i} $. The matrix $u$ is the fundamental representation with $u_{ij}=0$ for $i\neq j$ and $u_{ii}=z_i$. Moreover, it has the following universal property: \\
	If $(\mathcal{Q},\Delta)$ is a CQG generated by $k$ unitaries $\{x_{i}\}_{i=1}^{k}$ with $\Delta(x_{i})=x_{i} \otimes x_{i}$ for all $i \in \{1,2,...,k\}$, then there exists a surjective CQG morphism from $ (\underbrace{C(S^{1})* \cdots *C(S^{1})}_{k \textendash times}, \Delta)$ onto $(\mathcal{Q},\Delta)$, which sends $z_{i} \mapsto x_{i}$ for all $i \in \{1,2,...,k\}$ (see \cite{Wangfree}).

\end{enumerate}

\begin{rem}
    Recall that each element $g$ of $S_n$ can be identified with the permutation matrix $(g_{ij})_{n \times n}$ in $GL_n(\mathbb{C})$. 
	Let $(C(S_n), u)$ be the CMQG with the { \bf standard fundamental representation} $u=(u_{ij})_{n \times n}$, where $u_{ij}: S_n \to \mathbb{C}$ is defined by $u_{ij}(g)=g_{ij}$. Theorem 1.5 of \cite{Wor} asserts that $(C(S_n),u')$ is identical to $(C(S_n), u)$ for any fundamental representation $u'$ of order $n$. Therefore, if $(C(S_n),u')$ is a quantum subgroup of $(Q,q^{Q})$ with fundamental representation $q^{Q}=(q^{Q}_{ij})_{n \times n}$, then there exists a surjective $C^*$-homomorphism $\phi : Q \to C(S_n)$ such that $\phi(q^{Q}_{ij})=u_{ij}$. In this context, we simply say that \textbf{$S_n$ is a quantum subgroup of $(Q,q^{Q})$}.\\
\end{rem}

\subsection{Quantum automorphism groups}
In this section, we will discuss the CQG action and the quantum automorphism group of a $C^*$-algebra within a categorical framework. For more detailed information, readers can consult \cite{Wang} and \cite{Bichon}.
\begin{Def}
	A CQG $G$ acts faithfully on a unital $C^{*}$-algebra $\mathcal{C}$ if there exists a unital $C^{*}$-homomorphism $\alpha:\mathcal{C} \to \mathcal{C} \otimes C(G) $, called an action, which satisfies the following conditions: \vspace{0.1cm}
	\begin{itemize}
		\item[(i)]  Action equation: $(\alpha \otimes id_{C(G)})\alpha = (id_{\mathcal{C}} \otimes \Delta)\alpha.$ \vspace{0.1cm}
		\item[(ii)] Podle\'s condition: $span\{\alpha(\mathcal{C})(1\otimes C(G))\}$ is dense in $\mathcal{C} \otimes C(G).$ \vspace{0.1cm}
		\item[(iii)] Faithfulness: The $*$-algebra generated by the set $\{(\theta\otimes id)\alpha(\mathcal{C}) : \theta \in \mathcal{C}^{*}\}$ is norm-dense in $C(G)$.\vspace{0.1cm}
	\end{itemize}
\end{Def}
The pair $((C(G), \Delta),\alpha)$ is also referred to as the \textbf{quantum transformation group of $\mathcal{C} $}.\\

Given a unital $C^{*}$-algebra $ \mathcal{C} $, the \textbf{category of quantum transformation group of $\mathcal{C}$} is a \textbf{category} $ \mathfrak{C} $ which contains quantum transformation groups of $\mathcal{C}$ as objects, and a morphism from $((C(G_{1}),\Delta_{1}),\alpha_{1})$ to $((C(G_{2}),\Delta_{2}),\alpha_{2})$  is a CQG morphism $\phi :(C(G_{1}),\Delta_{1}) \to (C(G_{2}),\Delta_{2})$ such that $ (id_{\mathcal{C}} \otimes \phi)\alpha_{1}=\alpha_{2} $.\\
The \textbf{universal object of the category $ \mathfrak{C} $} is the quantum transformation group of $\mathcal{C}$, denoted by $((C(\widehat{G}),\widehat{\Delta}),\widehat{\alpha})$, which satisfies the following universal property:\\
For any object $((C(G), \Delta_{G}),\alpha_{G})$ in the category of quantum transformation group of $\mathcal{C}$, there exists a surjective CQG morphism $\widehat{\phi}: (C(\widehat{G}), \widehat{\Delta}) \to (C(G), \Delta_{G}) $ such that $(id_{\mathcal{C}} \otimes \widehat{\phi})\widehat{\alpha}=\alpha_{G} $. 

\begin{Def}
	Given a unital $C^{*}$-algebra $\mathcal{C}$, the {\bf quantum automorphism group} of $ \mathcal{C} $ is the underlying CQG of the universal object in the category of quantum transformation groups of $\mathcal{C}$, if the universal object exists.
\end{Def}

\begin{rem}
	An important observation is that the universal object might fail to exist in general when considering the above category. In 1998, Wang showed that although the universal object exists in the category of quantum transformation group of  $\mathbb{C}^n$ and the quantum automorphism group of $\mathbb{C}^n$ is $S_n^+$, the universal object in the category of quantum transformation group of $ M_{n}(\mathbb{C}) $  does not exist for $n \geq 2 $ (see \cite{Wang}). However, one can recover from this situation by restricting the category to a sub-category in the following way:\\
	Consider a linear functional $ \tau: \mathcal{C} \to \mathbb{C} $. Now, define a subcategory $\mathfrak{C}_{\tau}$ whose objects are quantum transformation groups $((\mathcal{Q},\Delta),\alpha)$ of $\mathcal{C}$ that satisfy the equation: $(\tau \otimes id)\alpha(.)=\tau(.).1 $ on a suitable subspace of $\mathcal{C}$, and morphisms are taken as described above.\\ 
	For example, in \cite{Wang}, the author assumed that each object in the category also preserves the state $Tr: M_{n}(\mathbb{C}) \to \mathbb{C}$ (the `Trace' map on $M_{n}(\mathbb{C})$), i.e., $(Tr \otimes id)\alpha(.)=Tr(.).1 $ on $M_{n}(\mathbb{C})$. Under this assumption, it was shown that the universal object exists in $\mathfrak{C}_{Tr}$ (see \cite{Wang} for more details). 
\end{rem}
\noindent {\bf Example:} (Quantum automorphism group of graphs) Recall that $S_n^+$ is a CMQG with a fundamental representation $u$. Let $\Gamma=\{V,E,s,r\}$ be a finite, directed graph without loops or multiple edges, and let $A_{\Gamma}$ denote its adjacency matrix. The quantum automorphism group of $\Gamma$ (in the sense of Banica), denoted by $QAut_{Ban}(\Gamma)$, is a CMQG whose underlying $C^*$-algebra is $C(S_{|V|}^+)/<uA_{\Gamma}-A_{\Gamma}u>$ (refer to \cite{Ban}, \cite{SSthesis} for details). \\
Let $K_n$ be the complete graph with $n$ vertices whose adjacency matrix $A_{K_n}$ is given by  $(A_{K_n})_{ij}= \begin{cases}
	0 & if ~ i=j\\
	1 & if ~ i\neq j\\
\end{cases}$, and let $K_n^c$ denote the complement graph of $K_n$, whose adjacency matrix is $0_{n \times n}$. If $\Gamma \cong {K_n} \text{ or } {K_n}^{c} $, then $QAut_{Ban}(\Gamma) \cong S_n^+$ (consult Example 2.1.8 of \cite{SSthesis}). 

\subsection{Quantum symmetry of a graph $C^{*}$-algebra}
Let $\Gamma=\{V(\Gamma),E(\Gamma),s,r\}$ be a finite graph without isolated vertices. Since $\Gamma$ has no isolated vertices, it suffices to define an action on the partial isometries corresponding to its edges. 
\begin{Def}(Definition 3.4 of \cite{Mandal})
	Given a graph $\Gamma$ without isolated vertices, a faithful action $\alpha$ of a CQG $G$ on a $C^{*}$-algebra $C^{*}(\Gamma)$ is said to be linear if $ \alpha(S_{e})=\sum\limits_{f \in E(\Gamma)} S_{f} \otimes q_{fe}$, where $q_{ef}\in C(G)$ for each $e,f \in E(\Gamma)$.\\
\end{Def}
 
Let us define
\begin{itemize}
	\item[(1)] $\mathcal{I}:=\{u\in V(\Gamma) : s^{-1}(u) = \emptyset\}$, \vspace{0.1cm}
	\item[(2)] $ E':= \{(e,f) \in E(\Gamma) \times E(\Gamma) : S_{e}S_{f}^{*} \neq 0 \}=\{(e,f) \in E(\Gamma) \times E(\Gamma) : r(e)=r(f) \}$. \vspace{0.1cm}
\end{itemize}

\noindent Lemma 3.2 of \cite{Mandal} shows that $\{ p_{u}, S_{e}S_{f}^{*} : u \in \mathcal{I}, (e,f) \in E' \}$ is a linearly independent set.\\[0.2cm]
Now, consider the vector subspace $\mathcal{V}_{2,+}= span\{ p_{u}, S_{e}S_{f}^{*} : u \in \mathcal{I}, (e,f) \in E' \}$ and a linear functional $\tau: \mathcal{V}_{2,+} \to \mathbb{C}$ defined by $\tau(S_{e}S_{f}^{*})=\delta_{ef}$, $\tau(p_{u})=1 $ for all $(e,f) \in E' $ and $ u \in \mathcal{I}$
(see subsection 3.1 of \cite{Mandal}).\\[0.1cm]
Since $\alpha(\mathcal{V}_{2,+}) \subseteq \mathcal{V}_{2,+} \otimes C(G)$ (by Lemma 3.6 of \cite{Mandal}),
 the equation $(\tau \otimes id)\alpha(.)=\tau(.).1 $ on $\mathcal{V}_{2,+}$ makes sense.\\

\begin{Def}(Definition 3.7 of \cite{Mandal})
	For a finite graph $\Gamma $ without isolated vertices, define a category $\mathfrak{C}_{\tau}^{Lin} $ whose objects are quantum transformation groups $((C(G),\Delta),\alpha)$ of the graph $C^*$-algebra $ C^{*}(\Gamma)$ such that $(\tau \otimes id)\alpha(x)=\tau(x).1 $ for all $x \in \mathcal{V}_{2,+}$. A morphism from $((C(G_1),\Delta_{1}),\alpha_{1})$ to $((C(G_2),\Delta_{2}),\alpha_{2})$ is given by a CQG morphism $\phi:C(G_{1}) \to C(G_{2}) $ such that $ (id_{C^{*}(\Gamma)} \otimes \phi)\alpha_{1}=\alpha_{2} $.\\
\end{Def}

\noindent Let $F^{\Gamma}$ be a $ (|E(\Gamma)| \times |E(\Gamma)| ) $ matrix defined by $(F^{\Gamma})_{ef}= \tau(S_{e}^{*}S_{f})$. In the following lemma, the matrix $F^{\Gamma}$ has been characterized. For the proof, we refer to Lemma 4.3 from \cite{Mandal}.   
\begin{lem} \label{Fgamma}
$F^{\Gamma}$ is an invertible diagonal matrix with $(F^{\Gamma})_{ee}= \begin{cases}
|s^{-1}(r(e))| & if ~ s^{-1}(r(e)) \neq \emptyset\\
1 & if ~ s^{-1}(r(e)) = \emptyset.
\end{cases}$
\end{lem}  
\noindent Therefore, $A_{u^{t}}(F^{\Gamma})$ is a CMQG. We refer the readers to Proposition 3.8 and Theorem 3.9 of \cite{Mandal} for the proof of the following theorem.
         
\begin{thm} \label{surjhom}
	For a finite graph $\Gamma$ without isolated vertices,\vspace{0.1cm}
	\begin{enumerate}
		\item there is a surjective CQG morphism from $A_{u^{t}}(F^{\Gamma})$ to any object in category $\mathfrak{C}_{\tau}^{Lin}$. \\
		\item the category $\mathfrak{C}_{\tau}^{Lin} $ admits a universal object.\\
	\end{enumerate}
	We denote the underlying CMQG and the corresponding action of the universal object by $ Q_{\tau}^{Lin}(\Gamma) $ and $\alpha^{univ}$ respectively, in category $\mathfrak{C}_{\tau}^{Lin} $. From the proof of Theorem 3.9 of \cite{Mandal}, it is evident that $ Q_{\tau}^{Lin}(\Gamma) $ is essentially a CMQG with the fundamental representation $q^{univ}$ such that $\alpha^{univ}(S_e)=\sum\limits_{f \in E(\Gamma)} S_f \otimes q^{univ}_{fe}$. We sometimes refer to the CMQG by $ (Q_{\tau}^{Lin}(\Gamma), q^{univ} )$.  \\
\end{thm}

\section{Main Results}
\noindent In this section, we identify the compact matrix quantum groups that can appear as the quantum symmetries of graph $C^*$-algebras having maximal permutational symmetry. First, we briefly discuss the analogous scenario in the context of the quantum symmetry of a graph in the next proposition.

\begin{prop} \label{unitaryeasygraph}
For a directed graph $\Gamma$ containing $n$ vertices, without loops or multiple edges, the following are equivalent:
\begin{enumerate}
	\item $S_n$ is a quantum subgroup of $QAut_{Ban}(\Gamma)$.
	\item $\Gamma$ is isomorphic to either $K_n$ or $K_n^c$.
	\item $QAut_{Ban}(\Gamma) \cong S_n^+$.
\end{enumerate}
\end{prop}
\begin{proof}
(1) $\implies$ (2): Since $S_n$ is a quantum subgroup of $(QAut_{Ban}(\Gamma),q)$, there exists a surjective CQG morphism  $\theta: (QAut_{Ban}(\Gamma),q) \to (C(S_n),u)$,  where $q$ and $u$ are fundamental representations of $QAut_{Ban}(\Gamma)$ and $C(S_n)$ respectively, both of order $n$. Consider $J:=<q_{ij}q_{kl}-q_{kl}q_{ij}>$, the commutator ideal of $QAut_{Ban}(\Gamma)$. It is clear that $J \subseteq ker\theta$. Hence, there exists a unique surjective CQG morphism $\tilde{\theta}: QAut_{Ban}(\Gamma)/J \to C(S_n) $ such that $\tilde{\theta}\pi=\theta$  (by Theorem 2.11 of \cite{Wangfree}). Since $QAut_{Ban}(\Gamma)/J \cong C(Aut(\Gamma))$ (see Lemma 2.1.6 of \cite{SSthesis}), there exists a surjective CQG morphism $\bar{\theta}: C(Aut(\Gamma)) \to C(S_n)$. Therefore, we have an injective group morphism from $S_n$ to $Aut(\Gamma)$. Moreover, $Aut(\Gamma)$ is a subgroup of $S_n$. Hence, $Aut(\Gamma) \cong S_n$. Using a straightforward graph-theoretic argument, one can deduce that there are only two graphs (with $n$ vertices): $K_n$  and $K_n^c$, both of which have the full automorphism group $S_n$. Therefore, $\Gamma$ is isomorphic to either $K_n$ or $K_n^c$. \\
(2) $\implies$ (3): $QAut_{Ban}(K_n) \cong QAut_{Ban}(K_n^c)  \cong S_n^+$ (Example 2.1.8 of \cite{SSthesis}).\\
(3) $\implies$ (1): The result is evident, as $S_n$ is a quantum subgroup of $S_n^+$.\\
\end{proof}

\begin{rem}
	Under the hypothesis of Proposition \ref{unitaryeasygraph}, the quantum automorphism group of a graph $C^*$-algebra, as defined by Schmidt and Weber, denoted by $QAut_{SW}(\Gamma)$, coincides with $QAut_{Ban}(\Gamma)$ (refer to \cite{Web}). Hence, the above equivalent statements hold in the graph $C^*$-algebraic context if we replace $QAut_{Ban}(\Gamma)$ by $QAut_{SW}(\Gamma)$.\\
\end{rem}

Now, recall that the CMQG $(Q_{\tau}^{Lin}(\Gamma),q^{univ})$ is the underlying CQG of the universal object in the category $\mathfrak{C}_{\tau}^{Lin} $. 
From now on, for our convenience, {\bf we denote the fundamental representation $q^{univ}$ of $Q_{\tau}^{Lin}(\Gamma) $ simply by $q$}.

\begin{prop} \label{S}
	Let $\mathcal{S}:=\{ \Gamma : \Gamma \text{ contains no loops, and } r(e)=v ~~\forall e \in E(\Gamma) \text{ for some } v \in V(\Gamma)\}$. If $\Gamma \in \mathcal{S}$, then $(Q_{\tau}^{Lin}(\Gamma), q) \approx U_n^+$ where $n=|E(\Gamma)|$.
\end{prop}
\begin{proof}
	Let $\Gamma \in \mathcal{S}$ with $|E(\Gamma)|=n$ and $r(e)=v$ for all $e \in E(\Gamma)$. Since $F^{\Gamma}=I_{n}$, by Theorem \ref{surjhom}, there exists a surjective CQG morphism $\phi: C(U_n^+) \to Q_{\tau}^{Lin}(\Gamma) $ such that $\phi(u_{ij})=q_{ij}$.\\
	On the other hand, we need to show that there exists an action $\alpha$ such that $((C(U_n^+),\Delta), \alpha)$ is an object of $\mathfrak{C}_{\tau}^{Lin}$. To that end, define a linear action $\alpha: C^*(\Gamma) \to C^*(\Gamma) \otimes C(U_n^+) $ by $\alpha(S_e)= \sum\limits_{f \in E(\Gamma)} S_f \otimes u_{fe}$, where $(u_{ij})_{n \times n}$ is the fundamental representation of $U_n^+$. To show that $\alpha$ is a well-defined $C^*$- homomorphism, we will demonstrate that $\{\alpha(S_e)\}_{e \in E(\Gamma)}$ satisfies the defining relations of $C^*(\Gamma)$. Firstly, we have
	\begin{align*}
	\alpha(S_e)^* \alpha(S_e)=& \sum\limits_{g \in E(\Gamma)} S_g^* S_g \otimes u_{ge}^* u_{ge}\\
	&=\sum\limits_{g \in E(\Gamma)} p_v \otimes u_{ge}^* u_{ge}= p_v \otimes \sum\limits_{g \in E(\Gamma)} u_{ge}^* u_{ge}\\
	&= p_v \otimes 1 
	\end{align*}
	for all $e \in E(\Gamma)$. Thus, we can extend $\alpha$ on $\{p_u: u \in V(\Gamma)\}$ by defining $\alpha(p_v)=p_v \otimes 1$ and	$\alpha(p_u)= \alpha \left( \sum\limits_{s(f)=u} S_f S_f ^*\right) = \sum\limits_{s(f)=u} ~~ \sum\limits_{k,l \in E(\Gamma)} S_k S_l^* \otimes u_{kf}u_{lf}^*$ for any $u \in V(\Gamma)-\{v\}$.\\ Moreover, for $e \neq f$, we verify $$\alpha(S_e)^* \alpha(S_f)= \sum\limits_{g \in E(\Gamma)} S_g^* S_g \otimes u_{ge}^* u_{gf}=p_{v} \otimes \sum\limits_{g \in E(\Gamma)} u_{ge}^* u_{gf}=0 .$$
	These calculations prove the existence of $\alpha.$ Clearly, $\alpha$ satisfies the action equation and faithfulness condition. For Podle\'s condition, note that 
	\begin{align*}
	\sum_{f \in E(\Gamma)}\alpha(S_f)(1 \otimes u_{ef}^{*}) &= \sum_{g,f \in E(\Gamma)}(S_g \otimes u_{gf})(1 \otimes u_{ef}^{*})\\
	&= \sum_{g \in E(\Gamma)} \left( S_g \otimes \sum_{f \in E(\Gamma)}u_{gf} u_{ef}^{*} \right)= \sum_{g \in E(\Gamma)} S_g \otimes \delta_{ge}= S_e \otimes 1 
	\end{align*}
	and 
	\begin{align*}
	\alpha(1)(1 \otimes u_{rs})= (1 \otimes u_{rs}).
	\end{align*}
	Since $(S_e \otimes u_{rs})=(S_e \otimes 1)(1 \otimes u_{rs})$, the Podle\'s condition follows.
	 For the $\tau$-preserving condition, observe that $S_e S_f^* \neq 0$ for all $e,f \in E(\Gamma)$ and 
	\begin{align*}
		(\tau \otimes id)\alpha(S_e S_f^*)=& \sum\limits_{g,h \in E(\Gamma)} \tau(S_g S_h^*) u_{ge} u_{hf}^*=\sum\limits_{g,h \in E(\Gamma)} \delta_{gh} u_{ge} u_{hf}^*\\
		=& \sum\limits_{g \in E(\Gamma)}  u_{ge} u_{gf}^*=\delta_{ef}=\tau(S_e S_f^*).
	\end{align*}
	Moreover, for the unique vertex $v$ with $s^{-1}(v)=\emptyset$, we have
	\begin{equation}
	(\tau \otimes id)\alpha(p_v)=\tau(p_v).
	\end{equation}
	Therefore, $U_n^+$ along with the action $\alpha$ is an object of $\mathfrak{C}_{\tau}^{Lin}$. Hence, by universal property, there exists a surjective CQG morphism $\psi:  Q_{\tau}^{Lin}(\Gamma) \to C(U_n^+) $ such that $\psi(q_{ij})=u_{ij}$. 
\end{proof}
 \begin{Observation}
	Note that every $\Gamma \in \mathcal{S}$ is connected. \label{EN1}
\end{Observation}

\begin{prop} \label{I1}
Let $\mathcal{I}_1:=\{ \Gamma : \Gamma \text{ contains no loops, and for each } e \in E(\Gamma), ~r(e) \text{ is a sink of indegree 1} \}$.\\
If $\Gamma \in \mathcal{I}_1$, then $(Q_{\tau}^{Lin}(\Gamma), q) \approx SH_{n}^{\infty+}$, where $n=|E(\Gamma)|$.
\end{prop}

\begin{proof}
Let $\Gamma \in \mathcal{I}_1$ with $|E(\Gamma)|=n$. Since $S_e^* S_f=0$ for $e \neq f$, applying the action $\alpha^{univ}: C^*(\Gamma) \to C^*(\Gamma) \otimes Q_{\tau}^{Lin}(\Gamma) $, we get
\begin{align*}
\sum_{g \in E{(\Gamma)}} p_{r(g)} \otimes q_{ge}^* q_{gf}=0.
\end{align*}
Since $r(g)$ is a sink of indegree 1, it follows that $q_{ge}^* q_{gf}=0$ for all $g \in E(\Gamma)$ and $e \neq f$. Additionally, using the formula $\kappa(q_{ij})=q_{ji}^{*} ~~ \forall~~ i,j \in E(\Gamma)$ corresponding to the antipode $\kappa$, one can deduce that $q_{eg}^{*} q_{fg}=0$ for all $g \in E(\Gamma)$ and $e \neq f$.\\
Moreover, from the relation $S_e S_f^* =0$ whenever $e \neq f$ (Proposition \ref{properties} (iv)) and using the fact that $r$ is injective, we have
\begin{align*}
\sum_{g \in E{(\Gamma)}} S_g S_g^*\otimes q_{ge} q_{gf}^*=0.
\end{align*}
Therefore, $q_{ge} q_{gf}^*=0$ for all $g \in E(\Gamma)$, since the set $\{S_e S_f^{*} : (e,f) \in E(\Gamma) \times E(\Gamma), r(e)=r(f) \}$ is linearly independent. Furthermore, whenever $e \neq f$, $q_{eg} q_{fg}^{*} =0$ for all $g \in E(\Gamma)$. Hence, one can conclude that $(Q_{\tau}^{Lin}(\Gamma),q)$ is a quantum subgroup of $ SH_{n}^{\infty +}$. \\
Next, to show that $SH_{n}^{\infty +}$ is a quantum subgroup of  $(Q_{\tau}^{Lin}(\Gamma),q)$, we need to demonstrate that there exists an action $\alpha$ such that $((C(SH_{n}^{\infty +}), \Delta), \alpha)$ is an object of $\mathfrak{C}_{\tau}^{Lin}$ for any $\Gamma \in \mathcal{I}_1$ with $|E(\Gamma)|=n$.
We define a linear action $\alpha: C^*(\Gamma) \to C^*(\Gamma) \otimes C(SH_{n}^{\infty+}) $ by setting $\alpha(S_e)= \sum\limits_{f \in E(\Gamma)} S_f \otimes u_{fe}$. To prove that $\alpha$ is a well-defined $C^*$-homomorphism, we have to show that $\{\alpha(S_e)\}_{e \in E(\Gamma)}$ satisfies the defining relations of $C^*(\Gamma)$. Now,
\begin{align*}
\alpha(S_e)^* \alpha(S_e)=& \sum\limits_{g \in E(\Gamma)} p_{r(g)} \otimes u_{ge}^* u_{ge} 
\end{align*}
for all $e \in E(\Gamma)$. Since $r$ is injective, we can extend $\alpha$ by defining $\alpha(p_{r(e)})=\sum\limits_{g \in E(\Gamma)} p_{r(g)} \otimes u_{ge}^* u_{ge}$ and	$\alpha(p_u)= \alpha \left( \sum\limits_{s(f)=u} S_f S_f ^* \right) = \sum\limits_{s(f)=u} ~~ \sum\limits_{k \in E(\Gamma)} S_k S_k^* \otimes u_{kf}u_{kf}^*$ for any $u \in V(\Gamma)$ with $s^{-1}(u) \neq \emptyset$.\\ Moreover, for $e \neq f$, we have $$\alpha(S_e)^* \alpha(S_f)= \sum\limits_{g \in E(\Gamma)} S_g^* S_g \otimes u_{ge}^* u_{gf}=0 .$$
These prove the existence of $\alpha$. Also, one can verify that $\alpha$ satisfies the action equation, faithfulness, Podle\'s condition and the $\tau$-preserving property. As a result, $SH_{n}^{\infty+}$ together with the action $\alpha$ is an object of $\mathfrak{C}_{\tau}^{Lin}$, which implies $(Q_{\tau}^{Lin}(\Gamma),q) \approx SH_{n}^{\infty+}$.
\end{proof}

\begin{Observation}
	 It is worth mentioning that $\Gamma \in \mathcal{I}_1$ is connected if and only if there exists some $v \in V(\Gamma) $ such that $s(e)=v$ for all $e \in E(\Gamma)$, i.e.  $\Gamma $ is isomorphic to $So_n$ (see Figure \ref{Son}). \label{EN2} \\
\end{Observation}

\begin{lem} \label{EL1}
	If $\Gamma$ is a graph without isolated vertices, containing a loop $l$ and an edge $e$ such that $ s(e) \neq r(e) $, then $S_{|E(\Gamma)|}$ cannot be a quantum subgroup of $(Q_{\tau}^{Lin}(\Gamma), q)$.
\end{lem}
\begin{proof}
 Since $ s(e) \neq r(e) $, we have $S_e S_e=0$ [by Proposition \ref{properties} (iii)]. Applying the linear action $\alpha^{univ}$, we get
 $$ \sum_{g,h \in E{(\Gamma)}} S_g S_h \otimes q_{ge} q_{he} =0 .$$
 Since $S_l S_l \neq 0$ (by Proposition \ref{properties} (iii)), using Lemma \ref{independent}, we can conclude that $q_{le}^{2}=0$.\\
 Now, suppose for a sake of contradiction that $S_n$ is a quantum subgroup of $(Q_{\tau}^{Lin}(\Gamma), q)$, then there exists a surjective CQG morphism $\phi: Q_{\tau}^{Lin}(\Gamma) \to C(S_n) $ such that $q_{ij} \to u_{ij}$, where $n=|E(\Gamma)|$ and $u=(u_{ij})_{n \times n}$ is the standard fundamental representation of $C(S_n)$.\\
 Hence, we have $\phi(q_{le}^{2})=u_{le}^{2}=0$, which implies $u_{le}=0$ (as $u_{le}$ is a projection). This leads to a contradiction.\\
\end{proof}

\begin{prop} \label{EP1}
	Let $\Gamma$ be a graph without isolated vertices, containing at least one loop. Then $S_n$ is a quantum subgroup of $(Q_{\tau}^{Lin}(\Gamma), q)$ if and only if $\Gamma$ is isomorphic to either $L_n$ (Figure \ref{Ln}) or $\sqcup_{i=1}^{n} L_1$ (Figure \ref{union L1}), where $n=|E(\Gamma)|$. 
\end{prop}

\begin{figure}[htpb]
	\begin{tikzpicture}
	\draw[fill=black] (0,0) circle (2pt) node[anchor=south]{$v_{1}$};
	\draw[black,thick](0,0.5) circle (0.5) node[above=0.25]{\rmidarrow}node[above=0.5]{$e_{11}$};
	\draw[black,thick](0,0.75) circle (0.75) node[above=0.5]{\rmidarrow} node[above=0.75]{$e_{22}$};
	\draw[black,thick](0,1.25) circle (1.25) node[above=1]{\rmidarrow}node[above=1.25]{$e_{nn}$};
	\draw[loosely dotted, black,thick] (0,2)--(0,2.5);
	\end{tikzpicture}
	\caption{$L_{n}$} \label{Ln}
\end{figure}

\begin{figure}[htpb]
\begin{tikzpicture}
\draw[fill=black] (0,0) circle (2pt) node[anchor=south]{$v_{1}$};
\draw[black,thick](0,0.5) circle (0.5) node[above=0.25]{\rmidarrow}node[above=0.5]{${\scriptscriptstyle{e_1}}$};

\draw[fill=black] (4,0) circle (2pt) node[anchor=south]{$v_{2}$};
\draw[black,thick](4,0.5) circle (0.5) node[above=0.25]{\rmidarrow}node[above=0.5]{${\scriptscriptstyle{e_2}}$};

\draw[fill=black] (12,0) circle (2pt) node[anchor=south]{$v_{n}$};
\draw[black,thick](12,0.5) circle (0.5) node[above=0.25]{\rmidarrow}node[above=0.5]{${\scriptscriptstyle{e_n}}$};

\draw[loosely dotted, black,thick] (7,0)--(9,0);
\end{tikzpicture}
\caption{$\sqcup_{i=1}^{n} L_{1}$} \label{union L1}
\end{figure}

\begin{proof}
	From Proposition 4.12 of \cite{Mandal} and Proposition 4.2 of \cite{Mandalkms}, we know that $(Q_{\tau}^{Lin}(L_n),q) \approx U_n^+$ and $(Q_{\tau}^{Lin}(\sqcup_{i=1}^{n} L_1),q) \approx H_n^{\infty+}$ respectively. In both cases, $S_n$ is a quantum subgroup of $(Q_{\tau}^{Lin}(\Gamma), q)$.\\
	
	Conversely, if $S_n$ is a quantum subgroup of $(Q_{\tau}^{Lin}(\Gamma), q)$, then there exists a surjective CQG homomorphism $\phi: Q_{\tau}^{Lin}(\Gamma) \to C(S_n) $ such that $q_{ij} \to u_{ij}$, where $n=|E(\Gamma)|$ and $u=(u_{ij})_{n \times n}$ is the standard fundamental representation of $C(S_n)$. By Lemma \ref{EL1}, we can conclude that there does not exist any $e \in E(\Gamma)$ with $s(e) \neq r(e)$; that is, all the edges in $\Gamma$ must be loops. If $|V(\Gamma)|=1$, then clearly it follows that $\Gamma \cong L_n$. If  $|V(\Gamma)| \geq 2$, then there are at least two distinct vertices $v,w \in V(\Gamma)$. Now, for the sake of contradiction, we further assume that there exist two distinct edges $e_1, e_2 \in E(\Gamma)$ such that $ s(e_i)=r(e_i)=v$ for  $i=1,2 $ and another edge $f_1 \in E(\Gamma)$ such that $s(f_1)=w$. Since $S_{e_1}S_{f_1}=0$ [by Proposition \ref{properties} (iii)], we have
	\begin{equation*}
	\sum_{g,h \in E{(\Gamma)}} S_g S_h \otimes q_{ge_1} q_{hf_1} =0 .
	\end{equation*} 
	Using the fact $S_{e_1} S_{e_2} \neq 0$ [by Proposition \ref{properties} (iii)] and Lemma \ref{independent}, one can conclude that $q_{e_1 e_1}q_{e_2 f_1}=0$. Therefore, applying $\phi$, we obtain $u_{e_1 e_1} u_{e_2 f_1}=0$ in $C(S_n)$. Now, consider the permutation matrix $g=(g_{ef})_{n \times n}$ which permutes the columns $e_2$ with $f_1$ (i.e., $g \leftrightarrow (e_2 ~~ f_1)$). Observe that $u_{e_1 e_1} u_{e_2 f_1} (g) = 1 \neq 0 $, which is a contradiction. Therefore, our earlier assumption was incorrect. Consequently, for $|V(\Gamma)|\geq 2$, each vertex can contain at most one loop, i.e. $\Gamma \cong \sqcup_{i=1}^{n} L_1$.\\
\end{proof}

\begin{cor} \label{Ecor1}
Let $\Gamma$ be a graph without isolated vertices, containing at least one loop. Then $S_n$ is a quantum subgroup of $(Q_{\tau}^{Lin}(\Gamma), q)$ if and only if $(Q_{\tau}^{Lin}(\Gamma), q)$ is identical to either $U_n^+$ or $H_n^{\infty+}$.
\end{cor}   
\begin{proof}
    We know that $(Q_{\tau}^{Lin}(L_n),q) \approx U_n^+$ (Proposition 4.12 of \cite{Mandal}) and $(Q_{\tau}^{Lin}(\sqcup_{i=1}^{n} L_{1}),q) \approx H_n^{\infty+}$ (Proposition 4.2 of \cite{Mandalkms}). Therefore, the result follows using Proposition \ref{EP1}.\\
\end{proof}	
	
\begin{lem} \label{EL2}
	Let $\Gamma$ be a graph without isolated vertices or loops, and let $n:=|E(\Gamma)| \geq 3 $. If there exists a vertex $v \in V(\Gamma)$ which is neither a sink nor a rigid source, then $S_n$ cannot be a quantum subgroup of $(Q_{\tau}^{Lin}(\Gamma), q)$.
\end{lem}

\begin{proof}
	Since $s^{-1}(v)$ and $r^{-1}(v)$ both are non-empty sets, let $r^{-1}(v)=\{e_1, e_2,..., e_p \}$ and $s^{-1}(v)=\{f_1, f_2,..., f_q \}$, where $1 \leq p,q < n$.\\
	Since	$p_v= S_{e_1}^*S_{e_1}=S_{f_1}S_{f_1}^*+S_{f_2}S_{f_2}^*+\cdots +S_{f_q}S_{f_q}^* $, applying the linear action $\alpha^{univ}$ on both sides, we get\\
	\begin{equation*}
	\sum_{k \in E(\Gamma)}S_k^*S_k \otimes q_{ke_1}^* q_{ke_1}= \sum_{k,l \in E(\Gamma)} S_k S_l^* \otimes (q_{kf_1} q_{lf_1}^*+q_{kf_2} q_{lf_2}^*+\cdots + q_{kf_q} q_{lf_q}^*).
	\end{equation*}
	Multiplying both sides by $(S_{f_1}^* \otimes 1)$ from the left and $(S_{f_1} \otimes 1)$ from the right, we obtain the following:
	\begin{equation*}
	\sum_{k \in E(\Gamma)}S_{f_1}^* S_k^*S_k S_{f_1} \otimes q_{ke_1}^* q_{ke_1}= \sum_{k,l \in E(\Gamma)} S_{f_1}^*S_k S_l^* S_{f_1} \otimes (q_{kf_1} q_{lf_1}^*+q_{kf_2} q_{lf_2}^*+\cdots + q_{kf_q} q_{lf_q}^*).
	\end{equation*}
	Using (i) of Proposition \ref{properties}, we have
	\begin{align*}
	 & \sum_{k \in E(\Gamma)} {(S_kS_{f_1})}^* {(S_kS_{f_1})} \otimes q_{ke_1}^* q_{ke_1}=  S_{f_1}^* S_{f_1} \otimes (q_{f_1f_1} q_{f_1f_1}^*+q_{f_1f_2} q_{f_1f_2}^*+\cdots + q_{f_1f_q} q_{f_1f_q}^*) \\
	\Rightarrow & \sum_{k: r(k)=s(f_1)=v} {(S_kS_{f_1})}^* {(S_kS_{f_1})} \otimes q_{ke_1}^* q_{ke_1}= S_{f_1}^* S_{f_1} \otimes (q_{f_1f_1} q_{f_1f_1}^*+q_{f_1f_2} q_{f_1f_2}^*+\cdots + q_{f_1f_q} q_{f_1f_q}^*).
	\end{align*}
	Using (v) of Proposition \ref{properties}, we know that ${(S_kS_{f_1})}^* {(S_kS_{f_1})}=p_{r(f_1)}$. Thus, we can rewrite the above equation as:
	\begin{align*}
	 & ~~ p_{r(f_1)} \otimes (q_{e_1 e_1}^* q_{e_1 e_1} + q_{e_2 e_1}^* q_{e_2 e_1}+\cdots +q_{e_p e_1}^* q_{e_p e_1} )= p_{r(f_1)} \otimes (q_{f_1f_1} q_{f_1f_1}^*+q_{f_1f_2} q_{f_1f_2}^*+\cdots + q_{f_1f_q} q_{f_1f_q}^*)\\
	\Rightarrow & ~~  (q_{e_1 e_1}^* q_{e_1 e_1} + q_{e_2 e_1}^* q_{e_2 e_1}+\cdots +q_{e_p e_1}^* q_{e_p e_1} )= (q_{f_1f_1} q_{f_1f_1}^*+q_{f_1f_2} q_{f_1f_2}^*+\cdots + q_{f_1f_q} q_{f_1f_q}^*). 
	\end{align*}
	Now, if we assume that $S_n$ is a quantum subgroup of $(Q_{\tau}^{Lin}(\Gamma), q)$, then there exists a surjective $C^{*}$-homomorphism $\phi: Q_{\tau}^{Lin}(\Gamma) \to C(S_n) $ such that $q_{ij} \to u_{ij}$, where $u=(u_{ij})_{n \times n}$ is the standard fundamental representation of $C(S_n)$. Applying $\phi$ on both sides of the above equation, we have
	\begin{equation}
	(u_{e_1 e_1} + u_{e_2 e_1}+\cdots +u_{e_p e_1} )= (u_{f_1f_1}+u_{f_1f_2}+\cdots + u_{f_1f_q} ) ~~ \text{on} ~ C(S_n). \label{EL2eq}
	\end{equation}
	Next, choose the permutation $g= \begin{cases}
	(e_2 ~~ f_1) & if ~ q=1, p>1 \\
	(e_1 ~~ f_2) & if ~ q>1, p=1 \\
	(e_2 ~~ f_1) & if ~ q>1, p>1 \\
	(e_1 ~~ e) & if ~ q=1, p=1\\
	\end{cases}$\\
	where $e \in E(\Gamma)$ with $e \neq e_1, f_1$ (such an edge $e$ exists because  $n \geq 3$). Observe that by evaluating both sides of Equation \eqref{EL2eq} on the chosen $g$, we arrive at a contradiction.\\
\end{proof}

\noindent \begin{Observation}
	Let $\Gamma$ be a graph without isolated vertices or loops. Then there exists a vertex $v \in V(\Gamma)$ which is neither a sink nor a rigid source if and only if $\Gamma$ contains a path of length 2. \\
\end{Observation} 

\begin{lem} \label{EL3}
	Let $\Gamma$ be a graph without isolated vertices or loops, containing a sink $v \in V(\Gamma)$ with $indeg(v)=m \geq 2$, and let $n:=|E(\Gamma)| > m $. Then $S_n$ cannot be a quantum subgroup of $(Q_{\tau}^{Lin}(\Gamma), q)$.
\end{lem}
\begin{proof}
	Let $r^{-1}(v)=\{e_1, e_2,..., e_m\}$ and $e \in E(\Gamma)- r^{-1}(v)$. We know that $p_v=S_{e_1}^* S_{e_1} = S_{e_2}^* S_{e_2} $.
	Applying the action $\alpha^{univ}$ on both sides of $S_{e_1}^* S_{e_1} = S_{e_2}^* S_{e_2}$, we obtain the following:
	\begin{align*}
	&\sum_{k \in E(\Gamma)}S_k^* S_k \otimes (q_{ke_1}^* q_{ke_1} -q_{ke_2}^* q_{ke_2})=0\\
	\Rightarrow &\sum_{k \in E(\Gamma)}p_{r(k)} \otimes (q_{ke_1}^* q_{ke_1} -q_{ke_2}^* q_{ke_2})=0.\\
	\end{align*}
	Next, multiplying $(p_v \otimes 1)$ from the left, we get
	$$\sum_{\{k \in E(\Gamma) : r(k)=v\}} (q_{ke_1}^* q_{ke_1} -q_{ke_2}^* q_{ke_2})=0,$$
	
	\noindent i.e., 
	\begin{equation}
	q_{e_1 e_1}^* q_{e_1 e_1}+q_{e_2 e_1}^* q_{e_2 e_1}+\cdots + q_{e_m e_1}^* q_{e_m e_1}= q_{e_1 e_2}^* q_{e_1 e_2} + q_{e_2 e_2}^* q_{e_2 e_2}+ \cdots + q_{e_m e_2}^* q_{e_m e_2} .  \label{EL3eq1}
	\end{equation}
	\noindent Now, if we assume that $S_n$ is a quantum subgroup of $(Q_{\tau}^{Lin}(\Gamma), q)$, then there exists a surjective homomorphism $\phi: Q_{\tau}^{Lin}(\Gamma) \to C(S_n) $ such that $q_{ij} \to u_{ij}$, where $u=(u_{ij})_{n \times n}$ is the standard fundamental representation of $C(S_n)$. Applying the homomorphism $\phi$ on both sides of Equation \eqref{EL3eq1}, we have 
	\begin{equation*}
	u_{e_1 e_1}+u_{e_2 e_1}+\cdots + u_{e_m e_1}= u_{e_1 e_2}+ u_{e_2 e_2}+ \cdots + u_{e_m e_2} ~~\text{ on } C(S_n).
	\end{equation*}
	Next, we multiply both sides of the above equation by $(u_{e_2 e_2}u_{e_3 e_3}...u_{e_m e_m})$. This yields the following:
	\begin{equation}
	u_{e_1 e_1}u_{e_2 e_2}u_{e_3 e_3}...u_{e_m e_m}=u_{e_2 e_2}u_{e_3 e_3}...u_{e_m e_m}  ~~\text{ on } C(S_n). \label{EL3eq2}
	\end{equation}
	By choosing the permutation $g=(e_1 ~~ e)$ and evaluating both sides of Equation \eqref{EL3eq2} on it, we find that the left-hand side (LHS) equals 0, while the right-hand side (RHS) equals 1. This leads to a contradiction.\\
\end{proof}

\begin{prop}
	Let $\Gamma$ be a graph with no isolated vertices or loops such that  $|E(\Gamma)|=n$. Then $S_n$ is a quantum subgroup of $(Q_{\tau}^{Lin}(\Gamma), q)$ if and only if $(Q_{\tau}^{Lin}(\Gamma),q)$ is identical to one of $U_n^+$, $SH_n^{\infty+}$ or $H_2^{\infty+}$ ($H_2^{\infty+}$ can only appear whenever $n=2$). \label{EP2}
\end{prop}

\begin{proof}
	We divide the proof into two cases.\\	
	\noindent {\bf Case:1} 	There exists a vertex in $V(\Gamma)$ which is neither a rigid source nor a sink.\\
	If $n=|E(\Gamma)| \geq 3 $, then $S_n$ cannot be a quantum subgroup of $(Q_{\tau}^{Lin}(\Gamma), q)$ (by Lemma \ref{EL2}). Therefore, $S_n$ is a quantum subgroup of $(Q_{\tau}^{Lin}(\Gamma), q)$ only if $|E(\Gamma)|=2$.\\
	

    \begin{figure}[htb]
    	\begin{minipage}[c]{0.4\linewidth}
    		\centering
    		\begin{tikzpicture}
    		\draw[fill=black] (0,0) circle (2pt) node[anchor=north]{$v_{1}$};
    		\draw[fill=black] (2,0) circle (2pt) node[anchor=north]{$v_{2}$};
    		\draw[fill=black] (4,0) circle (2pt) node[anchor=north]{$v_{3}$};

    		\draw[black] (0,0)-- node{\rmidarrow}  node[above]{$e_{12}$} (2,0);
    		\draw[black] (2,0)-- node{\rmidarrow}  node[above]{$e_{23}$} (4,0);
    		\end{tikzpicture}
    		\caption{$P_{2}$} \label{P2}
    	\end{minipage} \hfill  	
   \begin{minipage}[c]{0.4\linewidth}
   	\centering
   	\begin{tikzpicture}
   	\draw[fill=black] (0,0) circle (2pt) node[anchor=east]{$v_{1}$};
   	\draw[fill=black] (2,0) circle (2pt) node[anchor=west]{$v_{2}$};

   	\draw[black,thick](0,0) edge[bend left= 60] node{\rmidarrow} node[above]{$e_{12}$} (2,0);
   	\draw[black,thick](2,0) edge[bend left= 60] node{\lmidarrow} node[below]{$e_{21}$} (0,0);
   	
   	\end{tikzpicture}
   	\caption{$C_2$} \label{C2}
   \end{minipage}
    \end{figure}

	\noindent If $|E(\Gamma)|=2$, then $\Gamma$ must be isomorphic to either $P_2$ or $C_2$ (see Figures \ref{P2}, \ref{C2}).  Note that $(Q_{\tau}^{Lin}(P_2),q) \approx (C(S^1)*C(S^1),u)$ (see Example 4.1 of \cite{Mandal}), which does not contain $S_2$ as a quantum subgroup. On the other hand, $(Q_{\tau}^{Lin}(C_2),q) \approx H_2^{\infty+} $ (Theorem 4.5 of \cite{Mandal}), which contains $S_2$ as a quantum subgroup.\\
	 Therefore, $S_2$ is a quantum subgroup of $(Q_{\tau}^{Lin}(\Gamma), q)$ if and only if  $\Gamma \cong C_2$ if and only if  $(Q_{\tau}^{Lin}(\Gamma),q) \approx H_2^{\infty+}$.\\
	
	\noindent {\bf Case:2} 	Every vertex in $V(\Gamma)$ is either a rigid source or a sink. Therefore, $r(e)$ is a sink for all $e \in E(\Gamma)$.\\
	$\bullet$ If $r(e)$ is a sink of indegree 1 for all $e \in E(\Gamma)$, then by Proposition \ref{I1}, we have $(Q_{\tau}^{Lin}(\Gamma),q) \approx SH_{n}^{\infty +}$, which contains $S_n$ as a quantum subgroup.\\
	$\bullet$ If $r(e)$ is a sink of indegree $m \geq 2$ for at least one $e \in E(\Gamma)$, then we have two possible cases: \\
	(i) $|E(\Gamma)|=m$ : In this case, $\Gamma \in \mathcal{S}$. Therefore,  $(Q_{\tau}^{Lin}(\Gamma),q) \approx U_m^+$ (by Proposition \ref{S}). \\  
	(ii) $|E(\Gamma)|>m$ : Using Lemma \ref{EL3}, we can conclude that $S_n$ is not a quantum subgroup of $(Q_{\tau}^{Lin}(\Gamma), q)$. \\
	Therefore, in (Case:2),  $S_n$ is a quantum subgroup of $(Q_{\tau}^{Lin}(\Gamma), q)$ if and only if  $\Gamma \in \mathcal{S} \cup \mathcal{I}_1$ if and only if  $(Q_{\tau}^{Lin}(\Gamma),q)$ is identical to one of $U_n^+$ or $SH_{n}^{\infty+}$. 
\end{proof}

\begin{thm} \label{ET1}
	Let $\Gamma$ be a graph without isolated vertices and containing $n$ edges. Then $S_n$ is a quantum subgroup of $(Q_{\tau}^{Lin}(\Gamma), q)$ if and only if $(Q_{\tau}^{Lin}(\Gamma),q)$ is identical to one of $U_n^+$, $H_n^{\infty+}$ or $SH_n^{\infty+}$. 
\end{thm}
\begin{proof}
	The proof follows directly from Corollary \ref{Ecor1} and Proposition \ref{EP2}.\\
\end{proof}

\begin{thm} \label{ET2}
	Let $\Gamma$ be a connected graph containing $n$ edges. Then $S_n$ is a quantum subgroup of $(Q_{\tau}^{Lin}(\Gamma), q)$ if and only if $(Q_{\tau}^{Lin}(\Gamma),q)$ is identical to one of $U_n^+$, $SH_n^{\infty+}$ or $H_2^{\infty+}$ ($H_2^{\infty+}$ can only appear whenever $n=2$).
\end{thm}

\begin{proof}
	Again, the proof is divided into two cases.\\
	{\bf Case:1} $\Gamma$ contains a loop. Since $\Gamma$ is connected, $S_n$ is a quantum subgroup of $(Q_{\tau}^{Lin}(\Gamma), q)$ if and only if $\Gamma$ is isomorphic to $L_n$ if and only if $(Q_{\tau}^{Lin}(\Gamma),q)$ is identical to $U_n^+$ (by Proposition \ref{EP1}).\\  
	\begin{figure}[htpb]

		\begin{minipage}[c]{0.4\linewidth}
			\centering
			\begin{tikzpicture}
			\draw[fill=black] (0,0) circle (2pt) node[anchor=north]{$v$};
			\draw[fill=black] (2,0) circle (2pt) node[anchor=north]{$v_{n}$};
			\draw[fill=black] (1.414,1.414) circle (2pt) node[anchor=north]{$v_{1}$};
			\draw[fill=black] (0,2) circle (2pt) node[anchor=south]{$v_{2}$};
			\draw[fill=black] (-1.414,1.414) circle (2pt) node[anchor=north]{$v_{3}$};
			
			\draw[black] (0,0)-- node{\rmidarrow}  node[below]{$e_{n}$} (2,0);
			\draw[black] (0,0)-- node{\Nwmidarrow}  node[right]{$e_{1}$} (1.414,1.414);
			\draw[black] (0,0)-- node{\Nmidarrow}  node[right]{$e_{2}$} (0,2);
			\draw[black] (0,0)-- node{\Nemidarrow}  node[above]{$e_{3}$} (-1.414,1.414);

			\draw [loosely dotted,thick,domain=150:340] plot ({1.5*cos(\x)}, {1.5*sin(\x)});
			\end{tikzpicture}
			\caption{$So_n$} \label{Son}
		\end{minipage}
	\end{figure}

	\noindent {\bf Case:2} $\Gamma$ contains no loops. Since $\Gamma$ is connected, $S_n$ is a quantum subgroup of $(Q_{\tau}^{Lin}(\Gamma), q)$ if and only if $\Gamma \in \{C_2, So_{n}\} \cup \mathcal{S} $ (by Proposition \ref{EP2}, Observation \ref{EN1} and Observation \ref{EN2}) if and only if $(Q_{\tau}^{Lin}(\Gamma),q)$ is identical to one of the following:  $H_2^{\infty+}$ (this case can only occur whenever $n=2$), $SH_n^{\infty+}$ or $U_n^+$.\\
\end{proof}

\noindent In the table below, we describe all possible graphs without isolated vertices whose associated quantum automorphism groups  $Q_{\tau}^{Lin}(\Gamma)$ contain the permutation groups $S_{|E(\Gamma)|}$,

\begin{table}[htpb] \label{Tab}
	\begin{center}
		\begin{tabular}{|c|c|c|}
			\hline
			Graph ($\Gamma$) & $Q_{\tau}^{Lin}(\Gamma)$ & Connectedness of $\Gamma$ \\
			\hline \hline 
			$L_n$ & $U_n^+$ & connected\\
			\hline
			$\sqcup_{i=1}^{n} L_1$ & $H_n^{\infty+}$ & not connected\\
			\hline
			$C_2$ & $ H_2^{\infty+} $ &  connected\\
			\hline
			$\Gamma \in \mathcal{S}_{n}$ & $U_n^+$ &  connected\\
			\hline
			$\Gamma= So_n \in \mathcal{I}_{1}^{n}$ &  $SH_n^{\infty+} $ &  connected\\
			\hline
			$\Gamma\in \mathcal{I}_{1}^{n}$ but $\Gamma \ncong So_n$ &  $SH_n^{\infty+} $ & not connected\\
			\hline
		\end{tabular}
	\end{center}
		\caption {Graph $\Gamma$ (without isolated vertices) for which $S_n \subset (Q_{\tau}^{Lin}(\Gamma),q)$} \label{Table}
\end{table}
\noindent where $\mathcal{S}_n := \{\Gamma \in \mathcal{S} :  |E(\Gamma)|=n\}$ and  $\mathcal{I}_1^n := \{\Gamma \in \mathcal{I}_1 :  |E(\Gamma)|=n\}$. Thus, the only graph $C^*$-algebras $C^*(\Gamma)$ associated with the graphs $\Gamma$ listed in the table have maximal permutational symmetry.\\

\begin{rem}
	For a finite graph $\Gamma$ without isolated vertices and containing $n$ edges, if $(Q_{\tau}^{Lin}(\Gamma),q)$ is identical to a unitary easy group (with a fundamental representation of order $n$), then $S_n$ is a quantum subgroup of $(Q_{\tau}^{Lin}(\Gamma),q)$ \cite{easy}. As a consequence, using Theorem \ref{ET1}, one can conclude that $U_n^+$, $H_n^{\infty +}$ and $SH_n^{\infty +}$ are all the possible families of unitary easy quantum groups that can be identical to $(Q_{\tau}^{Lin}(\Gamma),q)$ associated with finite graphs $\Gamma$ without isolated vertices. In particular, for a connected graph $\Gamma$, $(Q_{\tau}^{Lin}(\Gamma),q)$ is identical to one of the following unitary easy quantum groups: $U_n^+$, $H_2^{\infty+}$ and $SH_n^{\infty+}$ (by Theorem \ref{ET2}).
\end{rem}

\begin{rem} \label{Un+}
	Let $\Gamma$ be a finite graph without isolated vertices and let $|E(\Gamma)|=n$. Then $(Q_{\tau}^{Lin}(\Gamma),q)$ is identical to $U_{n}^{+}$ if and only if $\Gamma \in \mathcal{S}_n \cup \{L_n\}$, where $\mathcal{S}_n := \{\Gamma \in \mathcal{S} :  |E(\Gamma)|=n\}$.\\
\end{rem}

In the next theorem, we establish that though the CMQG $(Q_{\tau}^{Lin}(\Gamma),q)$ is a quantum subgroup of $(A_{u^t}(F^{\Gamma}),u)$, maximal quantum symmetry only appears as $U_n^+$ identically (i.e., whenever $F^{\Gamma}$ is an invertible scalar matrix).

\begin{thm}  \label{AUtF}
	There does not exist any finite graph $\Gamma$ without isolated vertices for which $(Q_{\tau}^{Lin}(\Gamma),q) \approx (A_{u^t}(F^{\Gamma}),u)$, where $F^{\Gamma}$ is not a scalar matrix.
\end{thm}
\begin{proof}
	Since $F^{\Gamma}$ is not a scalar matrix, $F^{\Gamma}$ contains two distinct diagonal entries $\lambda_1, \lambda_2 \in \mathbb{N}$. Without loss of generality, assume that $(F^{\Gamma})_{ee} =|s^{-1}(r(e))|=\lambda_1$ with $\lambda_1 \geq 2$ for some $e \in E(\Gamma)$. \\
	
	\noindent {\bf Case 1:} Suppose that $e$ is not a loop. The condition  $|s^{-1}(r(e))|=\lambda_1 \geq 2 $ implies that there exists an edge (possibly a loop) $g \neq e $ such that $r(e)=s(g)$. In other words, either $P_2$ (Figure \ref{P2}) or $T'$ (Figure \ref{T'}) is a subgraph of $\Gamma$. In this case, since $S_e S_e=0$, we obtain $\sum\limits_{g,h \in E{(\Gamma)}} S_g S_h \otimes q_{ge} q_{he} =0$, and since $eg$ is a path in $\Gamma$ (i.e., $S_e S_g \neq 0$), it follows that 
	\begin{equation}
	q_{ee} q_{ge}=0. \label{t2eq1}
	\end{equation}    \\

		\begin{figure}[htb]
			\begin{minipage}[c]{0.4\linewidth}
				\centering
				\begin{tikzpicture}
				\draw[fill=black] (0,0) circle (2pt) node[anchor=north]{$s(g)=r(e)$};
				\draw[fill=black] (2,0) circle (2pt) node[anchor=north]{$s(e)$};
				
				\begin{scope}[decoration={markings,
					mark=at position 0.25 with {\arrow[black,thick]{<}},
				}]
				\draw[black,postaction={decorate}, thick](0,0.5) circle (0.5)  node[above=0.6]{$g$};
				\end{scope}
				
				\begin{scope}[decoration={markings,
					mark=at position 0.50 with {\arrow[black,thick]{<}},
				}]
				\draw[black,postaction={decorate}, thick ](0,0)--node[above]{$e$}  (2,0);
				\end{scope}
				
				\end{tikzpicture}
				\caption{$T'$} \label{T'}
			\end{minipage} \hfill
			\begin{minipage}[c]{0.4\linewidth}
				\centering
				\begin{tikzpicture}
				\draw[fill=black] (0,0) circle (2pt) node[anchor=north]{$r(e)=s(g)$};
				\draw[fill=black] (2,0) circle (2pt) node[anchor=north]{$r(g)$};
				
				\begin{scope}[decoration={markings,
					mark=at position 0.25 with {\arrow[black,thick]{<}},
				}]
				\draw[black,postaction={decorate}, thick](0,0.5) circle (0.5)  node[above=0.6]{$e$};
				\end{scope}
				
				\begin{scope}[decoration={markings,
					mark=at position 0.50 with {\arrow[black,thick]{>}},
				}]
				\draw[black,postaction={decorate}, thick ](0,0)--node[above]{$g$}  (2,0);
				\end{scope}
				
				\end{tikzpicture}
				\caption{$T$} \label{T}
			\end{minipage}
		\end{figure}

	\noindent {\bf Case 2:} Let $e$ be a loop. Now, $|s^{-1}(r(e))|=\lambda_1 \geq 2 $ implies that there exists an edge (possibly a loop) $g \neq e $ such that $r(e)=s(g)$. In other words, either $\Gamma$ contains $T$ (Figure \ref{T}) as a subgraph or $\Gamma$ is isomorphic to $L_{\lambda_1}$ (see Figure \ref{Ln}). Clearly,  if $\Gamma \cong L_{\lambda_1}$, then $F^{L_{\lambda_1}}$ is a scalar matrix $\lambda_1 I_n$, which contradicts the assumption that $F^{\Gamma}$ is not a scalar matrix. Lastly, if $g$ is not a loop (i.e., $\Gamma$ contains $T$ as a subgraph), then $S_g S_g=0$ and $eg$ is a path in $\Gamma$, which again implies
	\begin{equation}
	q_{eg} q_{gg}=0. \label{t2eq2}
	\end{equation}    
	In both cases, neither \eqref{t2eq1} nor \eqref{t2eq2} holds in $A_{u^t}(F^{\Gamma})$ for any $F^{\Gamma}$. So, there does not exist any $C^*$-isomorphism mapping $q_{ef} \to u_{ef}$. Therefore,
	$(Q_{\tau}^{Lin}(\Gamma),q)$ is not identical to $ (A_{u^t}(F^{\Gamma}),u)$ for any non-scalar matrix $F^{\Gamma}$.	 
\end{proof}

\begin{rem}
All the quantum automorphism groups of graph $C^*$-algebras appearing in this article are kac type. Also, in \cite{rigidity}, we have encountered a large class of graphs having quantum symmetry $(Q_{\tau}^{Lin}(\Gamma),q) \approx (\underbrace{C(S^1)*C(S^1)* \cdots *C(S^1)}_{|E(\Gamma)|}, u)$, which is again kac type. Moreover, all the quantum groups that appeared in \cite{cuntz} are kac type as well. Now, it is natural to ask whether $(Q_{\tau}^{Lin}(\Gamma),q)$ is always kac type. We expect the answer to be positive, since we have not encountered any non-kac $(Q_{\tau}^{Lin}(\Gamma),q)$ yet. On the other hand, $ (A_{u^t}(F^{\Gamma}),u)$ is non-kac type for any non-scalar matrix $F^{\Gamma}$. So, though we have proved the last theorem (Theorem \ref{AUtF}) up to identical sense, we believe that it can be extended up to CQG isomorphism.\\[0.5cm]
\end{rem}

\section*{Acknowledgements}
The first author acknowledges the financial support from Department of Science and Technology, India (DST/INSPIRE/03/2021/001688). Both the authors also acknowledge the financial support from DST-FIST (File No. SR/FST/MS-I/2019/41). We would also like to thank the anonymous referee for his/her useful comments and suggestions on the earlier version of the paper.

\section*{Data availability}
Data sharing is not applicable to this article as no data sets were generated or analysed during the current study.

\section*{Declarations}
\textbf{Conflict of interest:} On behalf of all authors, the corresponding author states that there is no conflict of interest.

\begin{minipage}[l]{0.4\linewidth}
	Ujjal Karmakar\\
	Department of Mathematics\\
	Presidency University\\ College Street, Kolkata-700073\\
	Email: \email{mathsujjal@gmail.com}
\end{minipage} \hfill
\begin{minipage}[r]{0.4\linewidth} 
	Arnab Mandal\\
	Department of Mathematics\\
	Presidency University\\ College Street, Kolkata-700073\\
	Email: \email{arnab.maths@presiuniv.ac.in}
\end{minipage}
\end{document}